\documentclass[oneside,10pt]{amsart}
\usepackage{graphicx}
\usepackage[active]{srcltx}
 \makeatletter
\renewcommand*\subjclass[2][2000]{%
  \def\@subjclass{#2}%
  \@ifundefined{subjclassname@#1}{%
    \ClassWarning{\@classname}{Unknown edition (#1) of Mathematics
      Subject Classification; using '1991'.}%
  }{%
    \@xp\let\@xp\subjclassname\csname subjclassname@#1\endcsname
  }%
}
 \makeatother
\usepackage{enumerate,url,amssymb,  mathrsfs,yhmath}

\newtheorem{theorem}{Theorem}[section]
\newtheorem{lemma}[theorem]{Lemma}
\newtheorem*{lemma*}{Lemma}
\newtheorem{proposition}[theorem]{Proposition}

\theoremstyle{definition}

\newtheorem{example}[theorem]{Example}

\theoremstyle{remark}
\newtheorem{remark}[theorem]{Remark}

\numberwithin{equation}{section}

\def\XXint#1#2#3{{\setbox0=\hbox{$#1{#2#3}{\int}$}
\vcenter{\hbox{$#2#3$}}\kern-.5\wd0}}

\def\le{\leqslant}
\def\ge{\geqslant}
\begin{document}

\title{On J. C. C. Nitsche's type inequality for hyperbolic space $\mathbf{H}^3$} \subjclass{Primary 58E20}


\keywords{Harmonic mappings, Spherical annuli, Hyperbolic geometry}
\author{David Kalaj}
\address{ Faculty of Natural Sciences and
Mathematics, University of Montenegro, Cetinjski put b.b. 81000
Podgorica, Montenegro} \email{davidk@t-com.me}

\begin{abstract} Let $\mathbf H^3$ be the hyperbolic space identified with the unit ball $\mathbf{B}^3 =  \{x\in \mathbf{R}^3: |x| <  1\}$ with the
Poincar\'e metric $d_h$ and assume that
${\mathcal{A}}(x_0,p,q):=\{x: p<d_h(x,x_0)< q\}\subset \mathbf H^3$
is an hyperbolic annulus with the inner and outer radii
$0<p<q<\infty$. We prove that if there exists a proper hyperbolic
harmonic mapping between annuli ${\mathcal{A}}(x_0,a,b)$ and
${\mathcal{A}}(y_0,\alpha,\beta)$ in the hyperbolic space $\mathbf
H^3$, then $\beta/\alpha>1+\psi(a,b)$, where  $\psi$ is a positive
function.
\end{abstract} \maketitle


\section{Introduction}

\subsection{Background and statement of the main result} In this paper by $\mathbf{A}(a,b)$ we denote the annulus $\{x\in\mathbf{R}^n:a<|x|<b\}$ in the euclidean
space $\mathbf{R}^n$, $n\ge 2$.  The unit ball is defined by
$\mathbf{B}^n = \{x\in \mathbf{R}^n: |x| <  1\}$ and  the unit
sphere is defined by ${S}^{n-1} = \{x\in \mathbf{R}^n: |x| = 1\}$
(here $x=(x_1,\dots,x_n)$ and $|x|=\sqrt{\sum_{i=1}^n x_i^2}$ \,).
Fifty years ago J. C. C. Nitsche \cite{n}, studying the minimal
surfaces and inspired by radial harmonic mappings between annuli
asked a question whether the existence of a euclidean harmonic
mapping between circular annuli $\mathbf{A}(a,1)$ and
$\mathbf{A}(\alpha,1)$ in $\mathbf{R}^2$ is equivalent with the
simple inequality
\begin{equation}\label{nit}\alpha\le
\frac{2a}{1+a^2}.\end{equation} This question is answered recently
in positive  by Iwaniec, Kovalev and Onninen in \cite{conj}. The
Nitsche conjecture is deeply rooted in the theory of doubly
connected minimal surfaces. Some partial results have been obtained
previously by Lyzzaik \cite{Al}, Weitsman \cite{weit} and the author
\cite{israel}. On the other hand in \cite{h} and in \cite{kalaj} is
treated the same problem for the  harmonic mappings w.r.t hyperbolic
and Riemann metric in two-dimensional hyperbolic space and in
two-dimensional Riemann sphere respectively. In \cite{kalaj,jmaa}
the author treated the three-dimensional case and obtained an
inequality for euclidean harmonic mappings between annuli on
$\mathbf{R}^3$. The $n-$dimensional generalization of conjectured
inequality \eqref{nit} is $$\alpha\le \frac{2a}{n-1+a^n}$$ and is
inspired by radial harmonic mappings
\begin{equation}\label{f}f(x)=\left(\frac{1-a^{n-1}\alpha}{1-a^n}+\frac{a^{n-1}\alpha-a^n}{(1-a^n)|x|^n}\right)x\end{equation}
between annuli $\mathbf{A}(a,1)$ and $\mathbf{A}(\alpha,1)$ (c.f.
\cite{jmaa}). The last conjectured inequality for $n\ge 3$ remains
an open problem.

The existence of harmonic mappings between certain annuli in the two
dimensional euclidean space is deeply related to the existence of
minimizers of Dirichlet integral without boundary data for
differentiable mappings between annuli. Further it is shown in
\cite{ar} that the minimizer of Dirichlet integral w.r.t. euclidean
metric of certain deformations between annuli $\mathbf{A}(a,1)$ and
$\mathbf{A}(\alpha,1)$ is a homeomorphism if and only if the
inequality \eqref{nit} is satisfied.  In this case the minimizer is
the harmonic diffeomorphism given by \eqref{f} ($n=2$). See also
\cite{dist, gaven} for some generalization of the previous problem
to radial metrics. In multidimensional setting (when $n\ge 3$), the
minimization problem of Dirichlet energy without boundary data is
essentially different from the case $n=2$. It seems that in this
case the minimization of $n$-energy is more appropriate. Then the
appropriate Euler-Lagrange equations reads as a $n-$harmonic
equation. In \cite{io} Iwaniec and Onninen formulated a J. C. C.
Nitsche type inequality for $n-$harmonic mappings and shown that
under these inequality the absolute minimizers of Dirichlet energy
are radial $n-$harmonic mappings.  One of advantages of $n-$harmonic
mappings for $n\ge 3$ is that they are invariant under M\"obius
transformations of the space. This property the class of
$n-$harmonic mappings share with hyperbolic harmonic mappings.
Moreover hyperbolic harmonic mappings are invariant under M\"obius
transformations of the domain as well as of the image domain.

In this paper we consider hyperbolic harmonic mappings between
certain subsets of the hyperbolic space $\mathbf{H}^3$. Li and Tam
in \cite{invent, anals} established the existence and regularity of
proper hyperbolic harmonic mappings between $\mathbf{H}^n$ onto
$\mathbf{H}^m$, satisfying certain conditions in the ideal boundary
$\mathbf{S}^{n-1}$.

The purpose of this paper is to study J. C. C. Nitsche type problem
for the harmonic mappings between domains of the hyperbolic spaces
$\mathbf{H}^3$.

In this paper we prove the following theorem:
\begin{theorem}\label{peopeo} Let ${\mathbf A}=\mathbf{A}(a,b),$ ${\mathbf A}'=\mathbf{A}(\alpha,\beta)\subset \mathbf{B}^3$, $0<a<b<1$, $0<\alpha<\beta< 1$ be spherical annuli
endowed with the hyperbolic metric of the unit ball. If there exists
a proper hyperbolic harmonic mapping $u$ of $\mathbf{A}$ onto
$\mathbf{A}'$ then
\begin{equation}\label{ence}\frac{(1 - \alpha^2)^2}{4 \alpha (1 + \alpha^2)}\log\left[\frac{1+\beta}{1+\alpha}\frac{1-\alpha}{1-\beta}\right]
\geq \left(-1 + \frac{a}{b} +
\log\frac{b}{a}\right):\left(1+\log\frac{1-a^2}{1-b^2}\right).\end{equation}
\end{theorem}
Recall that a mapping $f : X \to Y$ between two topological spaces
is proper if and only if the preimage of every compact set in $Y$ is
compact in $X$.
\subsection{Hyperbolic harmonic mappings}
 In general, if the metrics of two non-compact
manifolds $M^m$ and $N^n$  are given locally by $$ds^2_M =
\sum_{i,j}^m g_{ij} dx^idx^j$$ and $$ds^2_N =
\sum_{\alpha,\beta=1}^n h_{ij} du^\alpha du^\beta $$ respectively,
then the energy-density function of a $C^1$ map $u: M\to N$ is
defined by $$e(u)(x) = \sum_{i,j=1}^m\sum_{\alpha,\beta=1}^n
g^{ij}(x)h_{\alpha\beta}(u(x))\frac{\partial u^\alpha}{\partial
x^i}\frac{\partial u^\beta}{\partial x^j},$$ and the total energy of
u is given by $$E(u) =\int_Me(u)(x)dx.$$ The harmonic-map equation
from $M$ into $N$, which is the Euler-Lagrange equation for critical
points of the total energy functional, can be written as
$$\Delta u^\alpha(x) =- \sum_{i,j=1}^m\sum_{\alpha,\beta=1}^n
\Gamma^{\alpha}_{\beta\gamma}(u(x))g^{ij}(x)\frac{\partial
u^\beta}{\partial x^i}\frac{\partial u^\gamma}{\partial x^j},$$ for
all $1 \le \alpha \le m$, where the $\Gamma^{\alpha}_{\beta\gamma}$
are the Christoffel symbols of $N$. Here $\Delta$ is
Laplace-Beltrami operator on $M^m$. We refer to the monograph
\cite{sy} for some important properties of harmonic mappings.

The hyperbolic space $\mathbf{H}^n$ is identified with
$\mathbf{B}^n=\{x\in \mathbf R^n: |x|<1\}$ with the Poincar\'e
metric tensor given by
$$ds^2_{\mathbf{H}}(x)=\frac{4|dx|^2}{(1-|x|^2)^2}$$ which in polar
coordinates can be written as
$$ds^2_{\mathbf{H}}(x)=4\frac{d\rho^2+\rho^2\sum_{i,j=2}^ng_{ij}d\eta^id\eta^j}{(1-\rho^2)^2}$$
where $\rho^2=|x|^2=\sum_{i=1}^n x_i^2$ and $\sum_{i,j=2}^n
g_{ij}d\eta^i d\eta^j$ is the standard metric tensor on the unit
sphere $S^{n-1}$.

 Polar coordinates are natural to use when
working with annuli. We associate with any point $\,x\in
\mathbf{B}^n_\circ\,$ a pair of polar coordinates
\begin{equation}
   (r,\omega )\,\in (0,1)\times \mathbf{S}^{n-1} \;\sim \;
    \mathbf B^n_\circ
\end{equation}
where $r = |x|\,$ is referred to as the radial distance and
$\,\omega = \frac{x}{|x|} \,$ as the
 spherical coordinate of $\,x\,$. Obviously $\,x =r\,  \omega\,$ and the volume element in
 polar coordinates reads as $\textrm dV(x)\,=\, r^{n-1}\, d r\,  d
 \mathcal{H}^{n-1}(\omega)$, where $\mathcal{H}^{n-1}$ is the
 $n-1$-dimensional Hausdorff surface measure.

If the polar coordinates of a point $x\in \mathbf{B}^n$ are
$(r,\omega)$, then the geodesic polar coordinates are
$(2\tanh^{-1}(|x|\},\omega)$ and they will be of crucial importance
for our approach.
In \cite{invent} Li and Tam computed the coefficients of tension
field of a mapping $$u(x)=r(x) \Theta(x)=r(x)(\theta^1(x),\dots,
\theta^n(x)):\mathbf{X}\to \mathbf{B}^n, \ \ \mathbf{X}\subset
\mathbf{B}^m$$ in polar coordinates as follows
\begin{equation}\begin{split}\label{lita}
\tau(u)^1&=\frac{(1-\rho^2)^2}{4}\Delta_0r\\&\ \ \ \
+\frac{1}{4}\bigg(2(m-2)(1-\rho^2)\rho\frac{\partial \rho}{\partial
\rho}\\&\ \ \ \ \ +
\frac{r(1-\rho^2)^2(2|\nabla_0r|^2-(1+r^2)\sum_{p,q=2}^nh_{pq}\left<\nabla_0\theta^p,\nabla_0\theta^q\right>)}{1-r^2}\bigg)\end{split}
\end{equation}
and
\begin{equation}\begin{split}\label{litapo}
\tau(u)^s&=\frac{(1-\rho^2)^2}{2}\left(\Delta_0
\theta^s+\sum_{p,q=2}^n\tilde\Gamma^s_{pq}\left<\nabla_0\theta^p,\nabla_0\theta^q\right>\right)\\&+
\frac{1-\rho^2}{2}\left((m-2)\rho\frac{\partial\theta^s}{\partial
\rho}+\frac{(1+r^2)(1-\rho^2)\left<\nabla_0r,\nabla_0\theta^s\right>}{r(1-r^2)}\right)\end{split}
\end{equation}
for $s\ge 2$. Here $\Delta_0$ denotes the euclidean Laplacian and
$\nabla_0$ denotes the euclidean gradient and the
$\tilde\Gamma^s_{pq}$ denote the Christoffel symbols with respect to
the standard metric tensor of $S^{n-1}$. The mapping $u$ is harmonic
if $\tau(u)^s=0$ for $1\le s\le n$. We will use only the relation
\eqref{lita}. This in particular mean that, the results of this
paper can be formulated in slightly more general setting.

The set of isometries of the hyperbolic space is a  Kleinian group
subgroup of all M\"obius transformations of the extended space
$\overline{\mathbf{R}}^n$ onto itself denoted by
$\mathbf{Conf}(\mathbf{B}^n)=\mathbf{Isom}(\mathbf{H}^n)$. We refer
to the Ahlfors' book \cite{al} for detailed survey to this class of
important mappings. We  recall that the hyperbolic metric $d_h(z,w)$
of the unit ball $\mathbf{B}^n$ is defined by
$$
\tanh \frac{d_h(x,y)}{2}= \frac{|x-y|}{[x,y]},
$$ where $[x,y]^2:=1+|x|^2|y|^2-2\left<x,y\right>$. In
particular $d_h(0,x)=2\tanh^{-1}(|x|)$.  Since the harmonicity is an
isometric invariant, we obtain the following proposition.
\begin{proposition}\label{eko}
If $u:\mathbf{X}\to \mathbf{Y}$ is a harmonic mapping between the
domains $\mathbf{X}$ and $\mathbf{Y}$ of the unit ball and $f,g\in
\mathbf{Conf}(\mathbf{B}^n)$, then $f\circ u\circ g$ is a harmonic
mapping between $g(\mathbf{X})$ and $f(\mathbf{Y})$.
\end{proposition}
Let $x_0\in \mathbf{H}^3$ and assume that $0<p<q<\infty$. Then the
set $\mathcal{A}(x_0,p,q):=\{x\in \mathbf{H}^3: p<d_h(x,x_0)<q\}$ is
called the hyperbolic annulus. Moreover
$\mathcal{A}(0,p,q)=\mathbf{A}(p',q')$, $p=\tanh\frac{p'}{2}$ and
$q=\tanh\frac{q'}{2}$. Having in mind the previous observation,
together with Proposition~\ref{eko} we obtain the following
reformulation of Theorem~\ref{peopeo}
\begin{theorem}\label{peo} Let ${\mathcal A}={\mathcal A}(x_0,a',b'),$ ${\mathcal A'}={\mathcal A}(y_0,\alpha',\beta')
\subset \mathbf{B}^3$, $0<a'<b'<\infty$, $0<\alpha'<\beta'< \infty$
be spherical annuli endowed by hyperbolic metric of the unit ball.
If there exists a proper hyperbolic harmonic mapping $u$ of
$\mathcal{A}$ onto $\mathcal{A}'$ then
\begin{equation}\label{beprim}\frac{\beta'}{\alpha'}\ge1+
\frac{\sinh(2\alpha')}{\alpha'}\frac{\log[\coth\frac{a'}{2}
\tanh\frac{b'}{2}]+\coth\frac{b'}{2}
 \tanh\frac{a'}{2}-1}{1+2
\log[\cosh\frac{b'}{2}\mathrm{sech}\,\frac{a'}{2}])}.\end{equation}
\end{theorem}
\begin{remark}
Since $\frac{\sinh(2\alpha')}{\alpha'}> 2$, from \eqref{beprim}, we
obtain the following  $$\frac{\alpha'}{\beta'}< \varphi(a',b')<1,
\text{ (c.f. \eqref{nit}, for $n=2$ and euclidean metric) }.$$
\end{remark}
\section{preliminary results}
For a matrix $A=\{a_{ij}\}_{i,j=1}^n$ we define the Hilbert-Schmidt
and geometric  norm as follows $$\|A\|_2=\sqrt{\sum_{i,j}^n
a_{ij}^2}\text{ and }\|A\|=\sup\{|A x|: |x|=1\},$$ respectively. Let
$\lambda_1^2\le \dots \le \lambda_n^2$  be the eigenvalues of the
matrix $A^T A$. Then $$\|A\|_2=\sqrt{\sum_{i}^n \lambda_{i}^2}\text{
and }\|A\|=\lambda_n$$ and $$\mathrm{det} A=\prod_{k=1}^n
\lambda_k.$$ We say that $A$ is $K$-quasiconformal, where $K\ge 1$,
if $\lambda_n\le K \lambda_1$.
\begin{lemma}\label{extra}There hold the sharp inequality
\begin{equation}\label{main}|Ax_1\times \dots \times A x_{n-1}|\le
\left[\frac{K^2}{1+(n-1)K^2}\right]^{(n-1)/2}
\|A\|_2^{n-1}|x_1\times\dots\times x_{n-1}|,\end{equation} where
$1\le K\le \infty $ is the constant of quasiconformality of $A$. If
$A$ is an orthogonal transformation, then in \eqref{main} we have
equality with $K=1$. If $A$ is singular, then $K=\infty$ and we make
use of the convention $\frac{K^2}{1+(n-1)K^2}=\frac{1}{n-1}$.
\end{lemma}
\begin{proof} Assume firstly that $A$ is a nonsingular matrix. Then
$A$ is $K-$quasiconformal for some $K<\infty$. Let $x_i=\sum_{j=1}^n
x_{ij} e_j$, $i=1,\dots n-1$. Then
$$A x_1\times \dots\times Ax_{n-1}=\sum_{\sigma}
\varepsilon_\sigma x_{1,\sigma_1}\dots x_{n-1\sigma_{n-1}}
Ae_1\times \dots\times Ae_{n-1}.$$ It follows that
\begin{equation}\label{adj}A x_1\times \dots\times Ax_{n-1} =\tilde A(  x_1\times
\dots\times x_{n-1}).\end{equation} Here $\tilde A$ is the adjugate
of $A$, which for nonsingular matrix $A$ satisfies the relation
$\tilde A=\det A \cdot A^{-1}$. As $A$ is $K$ quasiconformal,
$\tilde A$ is $K$ quasiconformal as well. Let $\lambda_1^2\le \dots
\le \lambda_n^2$ be the eigenvalues of the matrix $A^T A$. $A$ is
$K-$quasiconformal if and only if
\begin{equation}\label{lam}\frac{\lambda_n}{\lambda_1}\le K.\end{equation} From $\tilde A = \det A
\cdot A^{-1}$, it follows that
$$\tilde \lambda_k=\det A\cdot \frac{1}{\lambda_k}, \text{and}\,
\tilde\lambda_n\le \tilde\lambda_{n-1}\le \dots\le \tilde\lambda_1$$
and consequently
$$\frac{\tilde \lambda_1}{\tilde\lambda_n}\le K.$$
From \eqref{adj} we obtain \begin{equation}\label{adj1}|A x_1\times
\dots\times Ax_{n-1}| \le\|\tilde A \| \cdot |x_1\times \dots\times
x_{n-1}|.\end{equation} Furthermore
\begin{equation}\label{adj2}\|\tilde A \|=\tilde \lambda _1=\frac{\det
A}{\lambda_1}=\prod_{k=2}^n \lambda_k.
\end{equation}  On the other hand \begin{equation}\label{A}\|A\|_2=\sqrt{\sum_{k=1}^n
\lambda_k^2}.\end{equation} From G--A inequality we have
\begin{equation}\label{but}\begin{split}\frac{\prod_{k=2}^n
\lambda_k}{\left(\sqrt{\sum_{k=1}^n
\lambda_k^2}\right)^{n-1}}&\le\frac{1}{(n-1)^{(n-1)/2}}
\frac{\left(\sqrt{\sum_{k=2}^n
\lambda_k^2}\right)^{n-1}}{\left(\sqrt{\sum_{k=1}^n
\lambda_k^2}\right)^{n-1}}\\&=
\left(\frac{B}{(n-1)(B+\lambda_1^2)}\right)^{(n-1)/2},\end{split}\end{equation}
where $B=\sum_{k=2}^n \lambda_k^2$. Since $\lambda_k$ is an
increasing sequence, from \eqref{lam} we have
\begin{equation}\label{ine}\lambda_1^2\ge \lambda^2_k/K^2, \ \ \  k=2,\dots, n.\end{equation} Summing the inequalities \eqref{ine} we obtain
\begin{equation}\label{inea}\lambda^2_1\ge
\frac{B}{(n-1)K^2}.\end{equation} From \eqref{but}, \eqref{inea},
\eqref{adj2} and \eqref{A} we obtain $$\frac{\|\tilde
A\|}{\|A\|_2^{n-1}}\le
\left[\frac{K^2}{1+(n-1)K^2}\right]^{(n-1)/2}.$$ This in view of
\eqref{adj1}, completes the proof of inequality of  lemma. To show
the sharpness of the inequality, take $A(x)=(x_1,Kx_2,\dots, Kx_n)$.
Then $A$ is $K-$quasiconformal. Moreover $$|Ae_2\times \dots \times
A e_{n}|=K^{n-1}$$ and
$$\left[\frac{K^2}{1+(n-1)K^2}\right]^{(n-1)/2}
\|A\|_2^{n-1}|e_2\times\dots\times e_{n}|=K^{n-1}.$$ Since the set
of singular matrices is nowhere dense and closed subset of
$M_{n\times n}$, for a singular matrix $A$ there exists a sequence
of positive real numbers $\epsilon_k$ converging to zero such that
$A_k=A+\epsilon_k I$ is a nonsingular matrix, where $I$ is the
identity matrix. Moreover the constants of quasiconformality $K_k$
of $A_k$ tend to $\infty$. By applying the previous proof to $A_k$
we obtain the inequality \eqref{extra} for $K=\infty$. The
inequality \eqref{extra} is attained for $A(x)=(0,x_2,\dots,x_n)$.
\end{proof}

\begin{proposition}\label{onep}
Let $u$ be a $C^1$ surjection between the spherical rings
$\mathbf{A}(a,b)$ and $\mathbf{A}(\alpha,\beta)$, and let
$\Theta=(\theta^1,\dots,\theta^n)={u}/{|u|}$. Let $P^{n-1}$ be a
closed $n-1$ dimensional hyper-surface that separates the components
of the set $\mathbf{A}^C(a, b)$. Then
\begin{equation}\label{inequ}\int_{P^{n-1}}\|D\Theta\|_2^{n-1}d\mathcal{H}^{n-1}\geq
(n-1)^{\frac{n-1}{2}}\omega_{n-1},
\end{equation}
and
\begin{equation}\label{inequsec}
\int_{\mathbf{A}(a,b)}\|D\Theta\|_2^{n-1}dV\geq
(n-1)^{\frac{n-1}{2}}(b-a)\omega_{n-1},
\end{equation} where $\omega_{n-1}$ denote the measure of $S^{n-1}$ and
$D\Theta=\{\theta^j_{x_i}\}_{i,j=1}^n$ is the differential matrix of
$\Theta$. Moreover $d\mathcal{H}^{n-1}$ is the $n-1$-dimensional
Hausdorff surface measure and $dV$ is the volume element.
\end{proposition}
\begin{proof}
Let $K^{n-1}$ be an $n-1$-dimensional rectangle and let
$g:K^{n-1}\to P^{n-1}$ be a parametrization of $P^{n-1}$. Then the
function $\Theta \circ g$ is a differentiable surjection from
$K^{n-1}$ onto the unit sphere $S^{n-1}$. Then we have $$
\int_{K^{n-1}}D_{\Theta\circ g}dV\geq \omega_{n-1}.$$ (cf.
\cite[p.~245]{israel}). According to Lemma~\ref{extra} (for
$K=\infty$), we obtain
\[\begin{split}{D_{\Theta\circ g}(x)}&={\left|D\Theta(g(x))\frac{\partial
g(x)}{\partial x_1}\times \dots \times D\Theta(g(x))\frac{\partial
g(x)}{\partial x_{n-1}}\right|}\\&\leq (n-1)^{\frac{1-n}{2}}
\|D\Theta(g(x))\|_2^{n-1}{D_g(x)}.\end{split}\] Hence we obtain
$$(n-1)^{\frac{n-1}{2}}\omega_{n-1}\leq
\int_{K^{n-1}}\|D\Theta(g(x))\|_2^{n-1}D_g(x)dV(x)=
\int_{P^{n-1}}\|D\Theta(\zeta)\|_2^{n-1}d\mathcal{H}^{n-1}(\zeta).$$
Thus we have proved (\ref{inequ}). It follows that
$$\int_{\mathbf{A}(a,b)}\|D\Theta\|_2^{n-1}dV=
\int_{a}^{b}\left(\int_{S^{n-1}(0,t)}\|D\Theta\|_2^{n-1}\,d\mathcal{H}^{n-1}\right)\,\mathrm{d}t\geq
(n-1)^{\frac{n-1}{2}}(b-a)\omega_{n-1}.$$ The proof of the
proposition has been completed.
\end{proof}
\section{The proof of main result}
One of key formulas follows from following lemma
\begin{lemma}
Let $u(x)=r(x) \Theta(x):\mathbf{A}\to \mathbf{A}'$ be a hyperbolic
harmonic mapping between the domains $\mathbf{A}$ and $\mathbf{A}'$
of the hyperbolic space $\mathbf{B}^n$ and assume that
$R(x)=2\tanh^{-1}(r(x))$ and $\rho=|x|$. Then

\begin{equation}\label{hyplap}\Delta_0 R+\frac{2(n-2)\rho}{(1-\rho^2)}\frac{\partial R}{\partial
\rho} =\frac{\sinh(2R)}{2}\|D \Theta\|_2^2,\end{equation} where
$\Delta_0$ and $\|D\Theta\|_2$ are euclidean Laplacian and
Hilbert-Schmidt norm of differential matrix respectively.
\end{lemma}
\begin{proof}
Since  $u$ is harmonic, from the equation  $\tau(u)^1=0$, where
$\tau(u)^1$ is defined in \eqref{lita}, we obtain
$$\Delta_0 r+\left(2(n-2)(1-\rho^2)^{-1}\rho\frac{\partial \rho}{\partial \rho} +
\frac{r(2|\nabla_0r|^2-(1+r^2)\sum_{p,q=2}^nh_{pq}\left<\nabla_0\theta^p,\nabla_0\theta^q\right>)}{1-r^2}\right)=0.$$
Let $g(q)=\tanh(q/2)$. Then $$\Delta_0 r= g''(R(x))|\nabla_0 R|^2 +
g'(R(x))\Delta_0 R$$ and $$|\nabla_0 r|^2=(g')^2|\nabla_0 R|^2.$$
Since
$$\frac{2g}{1-g^2}=\sinh(R)=-\frac{g''}{{g'}^2}$$ it follows that
$$\frac{1}{2} \mathrm{sech}^2\frac{R}{2}\Delta_0 R+(n-2) \mathrm{sech}^2\frac{R}{2}\frac{\rho}{1-\rho^2}\frac{\partial R}{\partial \rho}
=\frac{r((1+r^2)\sum_{p,q=2}^nh_{pq}\left<\nabla_0\theta^p,\nabla_0\theta^q\right>)}{1-r^2},$$ i.e.
$$\Delta_0 R+2(n-2)(1-\rho^2)^{-1}\rho\frac{\partial R}{\partial \rho}
=\frac{\sinh(2R)}{2}\sum_{p,q=2}^nh_{pq}\left<\nabla_0\theta^p,\nabla_0\theta^q\right>),$$
which can be written as
$$\Delta_0 R+2(n-2)(1-\rho^2)^{-1}\rho\frac{\partial R}{\partial \rho}
=\frac{\sinh(2R)}{2}\|D\Theta\|_2^2.$$
\end{proof}

\begin{proof}[Proof of Theorem~\ref{peo} (Theorem~\ref{peopeo})]
Let $\alpha'=2\tanh^{-1}\alpha$ and $\beta'=2\tanh^{-1}\beta$. Let
$\varphi_k :[\alpha',\beta'] \mapsto [\alpha',\beta'] $ be a
sequence of non decreasing functions, constant in some small
neighborhood of $\alpha'$, for example in
$[\alpha',\alpha'+(\beta'-\alpha')/k]$ and satisfying the following
conditions
\begin{equation}\label{posdif}
0\leq\varphi'_k(R)\to 1 \ \text{and} \ 0\leq\varphi''_k(R)\to 0 \
\text{as} \ k\to \infty
\end{equation}
$ \text{for every} \ R \in [\alpha',\beta'] .$ (See \cite{israel})
for an example of such sequence).  Let $R_k$ be a function defined
on $\{x:a<|x|< b \}$ by $R_k(x)=\varphi_{k}(R(x))$. Then
\begin{equation}\label{ariu}\Delta_0 R_k(x)=\varphi_k''(R(x))|\nabla
R(x)|^2+\varphi_k'(R(x))\Delta_0 R(x).\end{equation} Therefore
\[\begin{split}\Delta_0 R_k+2(n-2)(1-\rho^2)^{-1}\rho\frac{\partial
R_k}{\partial \rho}&=\varphi_k''(R(x))|\nabla
R(x)|^2+\varphi_k'(R(x))\Delta_0
R(x)\\&+\varphi_k'(R(x))2(n-2)(1-\rho^2)^{-1}\rho\frac{\partial
R}{\partial \rho}\\&=\varphi_k''(R(x))|\nabla
R(x)|^2+\varphi_k'(R(x))\frac{\sinh(2R_k)}{2}\|D\Theta_k\|_2^2.\end{split}\]
Thus \begin{equation}\label{deltak}\Delta_0
R_k+2(n-2)(1-\rho^2)^{-1}\rho\frac{\partial R_k}{\partial \rho}\ge
0\end{equation} for every $k$. By \eqref{ariu} and (\ref{posdif}) it
follows at once that
$$\Delta_0 R_k(x)\to \Delta_0  R(x) \ \text{as} \ k\to \infty$$ for
every $x\in \mathbf{A}(a,b)$. Similarly we obtain
$$\frac{\partial R_k}{\partial \rho}(x)\to \frac{\partial
R}{\partial \rho}(x) \ \text{as} \ k\to \infty$$ uniformly on $\{x:
|\zeta|=s\}$ for every $s\in (a,b)$.
 By applying Green's formula for $R_k$ on
$\{x:a \leq|x|\leq s \}$, we obtain
\begin{equation*}
\int_{|\zeta|=s}\frac{\partial R_k}{\partial
\rho}\,d\mathcal{H}^{n-1}(\zeta)- \int_{|\zeta|=a}\frac{\partial
R_k}{\partial \rho}\,d\mathcal{H}^{n-1}(\zeta) =\int_{a\leq|x|\leq s
}\Delta_0  R_k \, dV(x).
\end{equation*}
Since the function $R_k$ is constant in some neighborhood of the
sphere $|\zeta|=a$, it follows that for $a<s<b$ and large enough $k$
$$\int_{|\zeta|=s}\frac{\partial R_k}{\partial \rho}\,d\mathcal{H}^{n-1}(\zeta)=
\int_{a \leq|x|\leq s}\Delta_0  R_k \, dV(x). $$ Therefore
\[\begin{split}\int_{|\zeta|=s}\frac{\partial R_k}{\partial
\rho}\,d\mathcal{H}^{n-1}(\zeta)&+\int_{A_s}2(n-2)(1-\rho^2)^{-1}\rho\frac{\partial
R_k}{\partial \rho}dV(x)\\&=\int_{\alpha \leq|x|\leq
s}\left[\Delta_0 R_k+2(n-2)(1-\rho^2)^{-1}\rho\frac{\partial
R_k}{\partial \rho} \right]\, dV(x).\end{split}\] Further for
$\omega\in S^{n-1}$
\[\begin{split}\lim_{k\to\infty}\int_a^s\frac{\rho^{n}}{1-\rho^2}&\frac{\partial
R_k}{\partial \rho}(s\omega) d\rho\\&=
\lim_{k\to\infty}\left[\frac{s^{n}R_k(s\omega)}{1-s^2}-\frac{a^{n}R_k(a
\omega)}{1-a^2}-\int_{a}^s\frac{\rho^{n-1} (n + 2 \rho^2 - n
\rho^2)}{(1 - \rho^2)^2}R_k(\rho \omega)d\rho\right]
\\&\le \frac{s^{n}}{1-s^2}\beta'-\frac{a^{n}}{1-a^2}\alpha'-\int_{a}^s
\frac{\rho^{n-1} (n + 2 \rho^2 - n \rho^2)}{(1 - \rho^2)^2}\alpha'
d\rho\\&= \frac{s^{n}}{1-s^2}(\beta'-\alpha').\end{split}\]
Therefore
\begin{equation}\label{drilip}\begin{split}\int_{|\zeta|=s}\frac{\partial R}{\partial \rho}\,d\mathcal{H}^{n-1}(\zeta) &+
2(n-2)\int_{S^{n-1}}\frac{s^{n}}{1-s^2}(\beta'-\alpha')
d\mathcal{H}^{n-1}\\&\ge \limsup_{k\to\infty}\int_{a \leq|x|\leq
s}\bigg[\Delta_0 R_k+2(n-2)(1-\rho^2)^{-1}\rho\frac{\partial
R_k}{\partial \rho} \bigg]\, dV(x).
\end{split}\end{equation}
By applying Fatou's lemma, having in mind \eqref{deltak} and using
\eqref{hyplap}, letting $k \to \infty$, we obtain
\begin{equation}\label{drili}\begin{split}\limsup_{k\to\infty}\int_{a \leq|x|\leq s}\bigg[\Delta_0
R_k&+2(n-2)(1-\rho^2)^{-1}\rho\frac{\partial R_k}{\partial \rho}
\bigg]\, dV(x)\\&\ge \int_{a \leq|x|\leq s}\bigg[\Delta_0
R+2(n-2)(1-\rho^2)^{-1}\rho\frac{\partial R}{\partial \rho} \bigg]\,
dV(x)\\&=
\int_{S^{n-1}}\int_a^s\rho^{n-1}\frac{\sinh(2R)}{2}\|\nabla
\Theta\|_2^2 d\rho d\mathcal{H}^{n-1}.\end{split}\end{equation} From
\eqref{drilip} and \eqref{drili} we obtain
\begin{equation}\label{drilica}\begin{split}\int_{|\zeta|=s}\frac{\partial R}{\partial \rho}\,d\mathcal{H}^{n-1}(\zeta) &+
2(n-2)\int_{S^{n-1}}\frac{s^{n}}{1-s^2}(\beta'-\alpha')
d\mathcal{H}^{n-1}\\&\ge
\int_{S^{n-1}}\int_a^s\rho^{n-1}\frac{\sinh(2R)}{2}\|\nabla
\Theta\|_2^2 d\rho d\mathcal{H}^{n-1}.
\end{split}\end{equation}
It follows that
\[\begin{split}s^{n-1}\frac{\partial }{\partial s} \int_{|\zeta|=1}R(s\zeta)
\,d\mathcal{H}^{n-1}(\zeta)&+\int_{S^{n-1}}\frac{2(n-2)s^{n}}{1-s^2}(\beta'-\alpha')
d\mathcal{H}^{n-1}\\&\geq
\int_{S^{n-1}}\int_a^s\rho^{n-1}\frac{\sinh(2R)}{2}\|\nabla
\Theta\|_2^2 d\rho d\mathcal{H}^{n-1}\end{split}\] i.e.
\[\begin{split}s^{n-1}\frac{\partial
}{\partial s} \int_{|\zeta|=1}R(s\zeta)
\,d\mathcal{H}^{n-1}(\zeta)&+2(n-2)\omega_{n-1}\frac{s^{n}}{1-s^2}(\beta'-\alpha')
\\&\geq \int_{S^{n-1}}\int_a^s\rho^{n-1}\frac{\sinh(2R)}{2}\|\nabla
\Theta\|_2^2 d\rho d\mathcal{H}^{n-1}\end{split}\] or what is the
same
\[\begin{split}\frac{\partial }{\partial s} \int_{|\zeta|=1}R(s\zeta)
\,\mathrm{d}S(\zeta)&+2(n-2)\omega_{n-1}\frac{s}{1-s^2}(\beta'-\alpha')
\\&\geq
s^{1-n}\int_{S^{n-1}}\int_a^s\rho^{n-1}\frac{\sinh(2R)}{2}\|\nabla
\Theta\|_2^2 d\rho d\mathcal{H}^{n-1}.\end{split}\] Integrating the
previous expression w.r.t $s$ on $[a,b]$ we obtain
\begin{equation}\label{pomo}\begin{split}\omega_{n-1}(\beta'-\alpha')&+(n-2)\omega_{n-1}\log\frac{1-a^2}{1-b^2}(\beta'-\alpha')
\\&\geq \int_a^b
s^{1-n}\int_{S^{n-1}}\int_a^s\rho^{n-1}\frac{\sinh(2R)}{2}\|\nabla
\Theta\|_2^2 d\rho d\mathcal{H}^{n-1} ds\\&\ge
\frac{\sinh(2\alpha')}{2}\int_a^b
s^{1-n}\int_{S^{n-1}}\int_a^s\rho^{n-1}\|\nabla \Theta\|_2^2 d\rho
d\mathcal{H}^{n-1} ds.\end{split}\end{equation} Now we put $n=3$,
which implies this simple fact $n-1=2$. Combining now \eqref{pomo}
with Proposition~\ref{onep} we obtain
$$(\beta'-\alpha')+(n-2)\log\frac{1-a^2}{1-b^2}(\beta'-\alpha')
\geq {\sinh(2\alpha')}\int_a^b s^{-2} (s-a)ds$$
and therefore
\begin{equation}\label{pola}(\beta'-\alpha')\left(1+\log\frac{1-a^2}{1-b^2}\right) \geq
{\sinh(2\alpha')}\left(-1 + \frac{a}{b} +
\log\frac{b}{a}\right).\end{equation} By using the formulas
$a'=2\tanh^{-1}a$ and $b'=2\tanh^{-1}b$, we obtain \eqref{beprim},
for $x_0=0$ and $y_0=0$. The general case follows from
Proposition~\ref{eko}. By using the following formulas
$$2\tanh^{-1}t=\log\frac{1+t}{1-t}\text{   and   }
\sinh(4\tanh^{-1}t)=\frac{4 t (1 + t^2)}{(1 - t^2)^2},$$ we obtain
\eqref{ence}. This finishes the required proofs.
\end{proof}
\begin{example}
Assume that $$u(x)=r(\rho) \frac{x}{|x|}$$ is a hyperbolic harmonic
mapping. Then from \eqref{hyplap}, taking $\rho=e^t$,  we obtain
that $r(\rho)=\tanh\frac{y(t)}{2}$ where $y$ is a solution of the
differential equation
$$y''+(2-n)\coth(t) y'=\frac{(n-1)\sinh(2y)}{2}.$$ Two of many possible
solutions of the previous differential equation are
$y_{+}(t)=2\tanh^{-1}(e^{t})$ and $y_{-}(t)=2\tanh^{-1}(e^{-t})$.
The function $y_+$ produces the identity mapping $u_+(x)=x$, while
the function $y_{-}$ produces the inversion $u_{- }(x)=x/|x|^2$.
Both are hyperbolic harmonic mappings, but the second one maps the
complement of the unit ball onto the unit ball. Notice  that both,
the unit ball and its complement, with appropriate metrics, can
identify the hyperbolic space $\mathbf{H}^n$.

\end{example}


\begin{thebibliography}{99}

\bibitem{al} {\sc Ahlfors, L.} {\it M\"obius transformations in several
dimensions.} Ordway Professorship Lectures in Mathematics.
University of Minnesota, School of Mathematics, Minneapolis, Minn.,
1981. ii+150 pp.

\bibitem{ar}
{\sc Astala, K.; Iwaniec, T.; Martin, G.:} {\it Deformations of
annuli with smallest mean distortion.}  Arch. Ration. Mech. Anal.
{\bf 195} (2010), no. 3, 899--921.


\bibitem{h}
{\sc Han, Z.-C.}:  {\it Remarks on the geometric behavior of
harmonic maps between surfaces. Chow, Ben (ed.) et al., Elliptic and
parabolic methods in geometry.} Proceedings of a workshop,
Minneapolis, MN, USA, May 23--27,  1994. Wellesley, MA: A K Peters.
57-66 (1996).

\bibitem{io} {\sc Iwaniec, T.; Onninen,} J. \emph{$n$-Harmonic Mappings Between Annuli: The Art of Integrating Free Lagrangians,} to appear in Mem. Amer. Math. Soc.



\bibitem{conj} {\sc Iwaniec, T.,  Kovalev, L. K.,  Onninen, J.:}
\emph{The Nitsche conjecture}, J. Amer. Math. Soc. 24 (2011), no. 2,
345-373.
\bibitem{jmaa}
{\sc Kalaj, D.}:{\it On the univalent solution of PDE $\Delta u=f$
between spherical annuli.}  J. Math. Anal. Appl.  {\bf 327}  (2007),
no. 1, 1--11.
\bibitem{israel}
{\sc Kalaj, D.}: {\it On the Nitsche conjecture for harmonic
mappings in ${\Bbb R}\sp 2$ and ${\Bbb R}\sp 3$.}  Israel J. Math.
{\bf 150}  (2005), 241--251.

\bibitem{dist} {\sc Kalaj, D.}:
\emph{Deformations of annuli on Riemann surfaces with smallest mean
distortion.} ArXiv:1005.5269.


\bibitem{kalaj}
{\sc Kalaj, D.}: {\it Harmonic maps between annuli on Riemann
surfaces.} Israel J. Math. \textbf{182}, (2011), 123-147.
(arXiv:1003.2744).


\bibitem{gaven}

{\sc Martin, G., McKubre-Jordens, M.} \emph{Deformations with
smallest weighted Lp average distortion and Nitsche type phenomena.}
Journal of the London Mathematical Society, (in press).

\bibitem{invent}
\textsc{Li, P.,  Tam, L.-F.:} \emph{The heat equation and harmonic
maps of complete manifolds.}  Invent. Math. 105, No.1, 1-46 (1991).


\bibitem{anals}
\textsc{Li, P.,  Tam, L.-F.:} \emph{Uniqueness and regularity of
proper harmonic maps.} Ann. Math. (2) 137, No.1, 167-201 (1993).

\bibitem{Al}
{\sc Lyzzaik, A.:} {\it The modulus of the image of annuli under
univalent harmonic mappings and a conjecture of J. C. C. Nitsche.}
J. London Math. soc., (2) {\bf 64} (2001), pp. 369-384.


\bibitem{n}
{\sc Nitsche, J.C.C.}: {\it On the modulus of doubly connected
regions under harmonic mappings}, Amer. Math. Monthly {\bf 69}
(1962), 781-782.


\bibitem{sy}
{\sc Schoen, R.; Yau, S. T.:} {\it Lectures on harmonic maps.}
Conference Proceedings and Lecture Notes in Geometry and Topology,
II. International Press, Cambridge, MA, 1997. vi+394 pp.


\bibitem{vm}
{\sc Vuorinen, M.:} {\it Conformal geometry and quasiregular
mappings.} Lecture Notes in Mathematics, 1319. Springer-Verlag,
Berlin, 1988. xx+209 pp.

\bibitem{weit}
{\sc Weitsman, A.:} {\it Univalent harmonic mappings of annuli and a
conjecture of J.C.C. Nitsche.} Israel J. Math. {\bf 124},  327--331
(2001)

\end{thebibliography}
\end{document}